\documentclass[12pt]{amsart} 

 \usepackage[font=footnotesize]{caption}
 \usepackage{subcaption}

\usepackage[dvipsnames,usenames]{xcolor} 

\usepackage{amsfonts,latexsym,wasysym, mathrsfs, mathtools,hhline,color, bm, bbm}
\usepackage[all, cmtip]{xy}

\usepackage{url}

\definecolor{hot}{RGB}{65,105,225} 

\usepackage[pagebackref=true,colorlinks=true, linkcolor=hot ,  citecolor=hot, urlcolor=hot]{hyperref}

\usepackage{textcomp}
\usepackage{tipa}

\usepackage{graphicx,enumerate}
\usepackage{enumitem}
\usepackage[normalem]{ulem}

\usepackage{subcaption}

\usepackage{amsthm}  
\usepackage{amsmath} 
\usepackage{amsfonts}  
\usepackage{amssymb}
\usepackage{amscd} 
\usepackage[all]{xy} 

\usepackage{graphicx}
\usepackage{pdfsync}
\usepackage{url}
\usepackage{comment}
\usepackage{tikz}
\usetikzlibrary{arrows}
\usepackage[linesnumbered,ruled,vlined]{algorithm2e}
\usepackage{cleveref}

\theoremstyle{definition} 
\newtheorem{thm}{Theorem}[section]
\newtheorem{conj}[thm]{Conjecture}

\theoremstyle{definition}

\newtheorem{theorem}[thm]{Theorem}

\newtheorem{prop}[thm]{Proposition}

\newtheorem{question}[thm]{Question}

\numberwithin{equation}{section}

\newcommand{\RR}{\mathbb{R}}

\newcommand{\CC}{\mathbb{C}}

\title{Real circles tangent to 3 conics}

\author[Breiding]{Paul Breiding}
\address{
Paul Breiding \\
University of Osnabr\"uck \\ Osnabr\"uck, Germany
}
\email{pbreiding@uni-osnabrueck.de}
\urladdr{\url{https://pbrdng.github.io/}}

\author[Lindberg]{Julia Lindberg}
\address{
Julia Lindberg \\
Max-Planck Institute \\ Leipzig, Germany
}
\email{julia.lindberg@mis.mpg.de}
\urladdr{\url{https://sites.google.com/view/julialindberg/home}}

\author[Ong]{Wern Juin Gabriel Ong}
\address{
Wern Juin Gabriel Ong \\
Bowdoin College \\ Brunswick, Maine, USA
}
\email{gong@bowdoin.edu}
\urladdr{\url{https://wgabrielong.github.io/}}

\author[Sommer]{Linus Sommer}
\address{
Linus Arthur Sommer\\
University of Leipzig \\ Leipzig, Germany
}
\email{ls67rynu@studserv.uni-leipzig.de}

\begin{document}

\maketitle

\begin{abstract}
In this paper we study circles tangent to conics. We show there are generically $184$ complex circles tangent to three conics in the plane and we characterize the real discriminant of the corresponding polynomial system. We give an explicit example of $3$ conics with $136$ real circles tangent to them. We conjecture that 136 is the maximal number of real circles. Furthermore, we implement a hill-climbing algorithm to find instances of conics with many real circles, and we introduce a machine learning model that, given three real conics, predicts the number of circles tangent to these three conics.
\end{abstract}

\section{Introduction}
Problems of tangency have been of interest since early geometry. Apollonius, in the $`E\pi\alpha\phi\alpha\Acute{\iota}$ (\emph{Tangencies}), asked the following question: Given three circles in the plane, how many circles are tangent to all three? He showed that the answer is 8 in general. In $1848$ Steiner asked the related question of how many conics are tangent to five generic conics. Steiner conjectured that there are $6^5$ such conics, meeting the B\'ezout bound of the corresponding polynomial system. In $1859$ and $1864$ Jonqui\`eres and then Chasles established the correct answer of $3264$. In this paper we study a problem which sits in between: 
\begin{question}\label{ques:C_circles}
    Given three general conics $Q_1,Q_2,Q_3 \subseteq \mathbb{C}^2$, how many circles are tangent to all three conics?
    \end{question}

\begin{figure}
\begin{center}
\includegraphics[width = 0.9\textwidth]{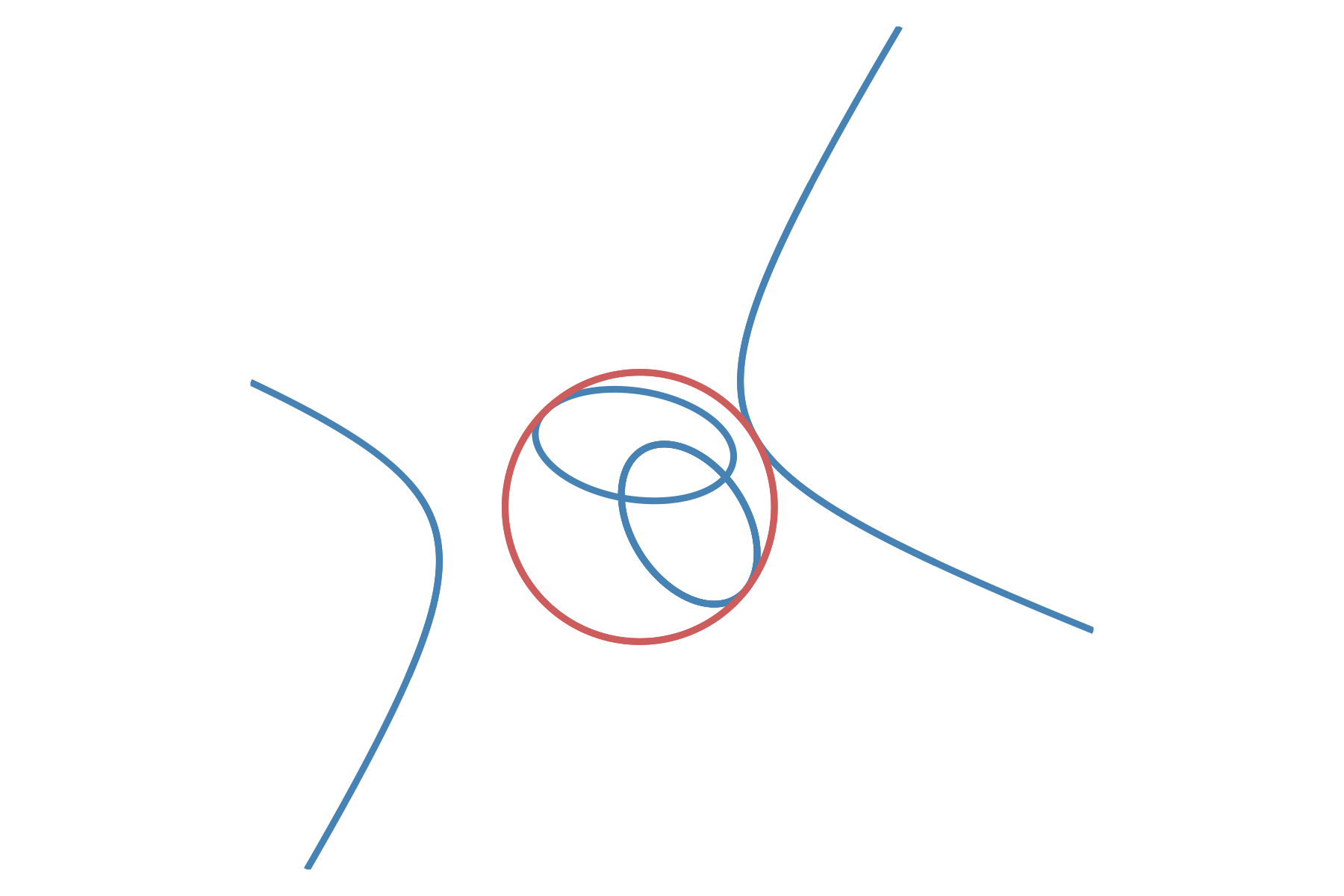}
\end{center}
\caption{A red circle tangent to one blue hyperbola and two blue ellipses.}
\end{figure}

Circles are defined by three numbers -- the coordinates of the center and the radius. Thus, we expect that the answer to our question is a finite number. Indeed, Emiris and Tzoumas \cite{ApollonianCircles} showed that there are \emph{at most} 184 circles tangent to three general conics. We show in the next section that this bound is attained for generic conics.

Next, we turn our attention to the real version of \Cref{ques:C_circles}.
By a real conic we mean a conic whose defining equation has real coefficients.
For three real circles the~8 Appolonius circles are real circles. The real version of Steiner's problem was studied only recently. In \cite{3264RealConics} the authors prove that there exist five real conics such that all $3264$ conics tangent to them are real. In \cite{breiding20203264}, the authors use numerical algebraic geometry to explicitly find such an instance and compute all $3264$ real conics. Such arrangements are called \emph{fully real}. In the same spirit we pose the following question.
\begin{question}\label{ques:R_circles}
    Given three general \emph{real} conics $Q_1,Q_2,Q_3 \subseteq \mathbb{R}^2$, how many \emph{real} circles are tangent to all three conics?
\end{question}
Here, we have the following result. 
\begin{theorem}\label{thm:136_R_intro}
There is an instance of three real conics $Q_1,Q_2,Q_3 \subseteq \mathbb{R}^2$, such that there are 136 real circles tangent to these three conics.
\end{theorem}
\Cref{ques:R_circles} is much more subtle than \Cref{ques:C_circles}. Of course, the answer to \Cref{ques:C_circles} gives an upper bound to \Cref{ques:R_circles} but it is non-trivial to verify whether or not that upper bound is tight.
In fact, we are not able to prove that 136 is the maximal number. In \cite{3264RealConics} the authors show that a fully real instance of Steiner' conic problem  exist in a neighborhood of the degenerate case where all five conics are double lines and each line intersects in the vertex of a regular pentagon. If we make the same construction for our circle problem, we find a maximum of $136$ real tritangent circles, not 184. The reason is that there are $4$ real conics that are tangent to $2$ lines and pass through $3$ points, but only $2$ real circles that are tangent to $2$ lines and pass through $1$ point \cite{CirclesAndSpheres} (circles are conics which pass through the two special \emph{circular points} $\circ_{+}:=[1:i:0], \circ_{-}:=[1:-i:0]$ in~$\mathbb{P}^{2}_{\mathbb{C}}$; see \Cref{circular_points}). This discrepancy is the reason why we get 136 instead of 184 real circles using the strategy from~\cite{3264RealConics}. This and computational evidence leads us to state the following conjecture.
\begin{conj}\label{conj_real}
The maximal number of real circles tangent to three conics is 136.
\end{conj}
As a first step towards proving \Cref{conj_real} we give insight to \Cref{ques:R_circles} by characterizing the \emph{real discriminant} of our tangency problem.

In the last part of the paper we approach \Cref{ques:R_circles}  computationally. In Section~\ref{sec:hillClimbing} we implement a \emph{hill climbing algorithm}, described in \cite{LinearProgram}, to find explicit conics that have many real tritangent circles. For instance, we use the algorithm to find for every even number $0\leq n\leq 136$ an arrangement of conics such they have exactly~$n$ real tritangent circles; see \Cref{thm:all are possible}. In addition, we introduce a machine learning model that, given three real conics $Q_1, Q_2, Q_3$, predicts the number of real circles tangent to these conics. We do this using supervised learning on training data generated with the help of the hill climbing algorithm.

Our code and all the data we generated is available on our \texttt{MathRepo} page
\medskip

\begin{center}
\url{ https://mathrepo.mis.mpg.de/circlesTangentConics}
\end{center}
\medskip

\subsection{Outline of paper}  In \Cref{sec:circles} we answer \Cref{ques:C_circles} by showing that for three general conics there are $184$ tritangent complex circles. We then classify the real discriminant of our tangency problem and show that there exists conics that have $136$ real tritangent circles. \Cref{sec:hillClimbing} outlines the hill-climbing algorithm, while \Cref{sec:ml} explores the application of machine learning to predicting the real solution count.

\subsection{Acknowledgements}
We thank Taylor Brysiewicz for explaining to us the basic principles of the hill climbing algorithm in \Cref{sec:hillClimbing}. We also thank Guido Mont\'{u}far, Rainer Sinn, and Bernd Sturmfels for helpful conversations. 

Breiding is funded by the Deutsche Forschungsgemeinschaft (DFG, German Research Foundation) -- Projektnummer 445466444.

Ong thanks Bowdoin College for support. 

This project was initiated in an REU program organized by Rainer Sinn at the MPI MiS Leipzig. We thank MPI MiS for providing space for collaboration.

\section{Real and complex circles tangent to three general conics}\label{sec:circles}
We begin by outlining the problem formulation under consideration. We work in an affine chart of $\mathbb{P}^{2}_{\mathbb{C}}$ that we identify with $\mathbb{C}^{2}$. Recall that a conic in the plane is the set of $(x,y) \in \CC^2$ satisfying the equation:
\begin{equation}\label{def_Q}
Q(x,y)=ax^{2}+bxy+cy^{2}+dx+ey+f = 0
\end{equation}
where $a,b,c,d,e,f \in \CC$ and a circle of radius $r$ centered at $(s,t) \in \CC^2 $ is given by~$(x,y) \in \CC^2$ that satisfy:
$$C(x,y)= (x-s)^{2}+(y-t)^{2}-r^{2} = 0.$$
The conic and circle intersect in 4 points, counting multiplicity and including points at infinity, so long as $Q$ and $C$ are irreducible and distinct. 
A point $(x,y)\in\mathbb{C}^{2}$ satisfying the two equations $Q(x,y)=C(x,y)=0$ is a point of tangency if and only if it has multiplicity at least two, or equivalently that the determinant of the Jacobian of $Q$ and $C$ vanishes:
$$
\det \big( [\nabla Q(x,y) \ \ 
\nabla C(x,y)] \big) = 0.$$
Here, 
$\nabla Q(x.y) = (\tfrac{\partial Q}{\partial x},\ \tfrac{\partial Q}{\partial y})^T$
denotes the \emph{gradient} of $Q$.
We denote this as 
$$\nabla Q(x,y) \wedge \nabla C(x,y) := \det \big( [\nabla Q(x,y) \ \ 
\nabla C(x,y)] \big).$$
This allows us to rephrase the tritangent circles problem as the set of solutions of a polynomial system. Let fix three conics 
\begin{align*}
    Q_{1}(x,y) &= a_{1}x^{2} + a_{2}xy + a_{3}y^{2} + a_{4}x + a_{5}y + a_{6} \\
    Q_{2}(x,y) &= b_{1}x^{2} + b_{2}xy + b_{3}y^{2} + b_{4}x + b_{5}y + b_{6} \\
    Q_{3}(x,y) &= c_{1}x^{2} + c_{2}xy + c_{3}y^{2} + c_{4}x + c_{5}y + c_{6}
\end{align*}
Let $(u_{i},v_{i})$ be the points of tangency on $Q_{i}$ (defined by $f_{1}, f_{2}, f_{3}$). A circle $C$ is tangent to all three of $Q_{1},Q_{2},Q_{3}$ if and only if the following conditions holds. First, $(u_{i},v_{i})\in Q_i$ for $1\leq i\leq 3$. This is formulated by the following three polynomials
\begin{align*}
    f_{1} &= a_{1}u_{1}^{2} + a_{2}u_{1}v_{1} + a_{3}v_{1}^{2} + a_{4}u_{1} + a_{5}v_{1} + a_{6}\\\nonumber
    f_{2} &= b_{1}u_{2}^{2} + b_{2}u_{2}v_{2} + b_{3}v_{2}^{2} + b_{4}u_{2} + b_{5}v_{2} + b_{6}\\\nonumber
    f_{3} &= c_{1}u_{3}^{2} + c_{2}u_{3}v_{3} + c_{3}v_{3}^{2} + c_{4}u_{3} + c_{5}v_{3} + c_{6}
\end{align*}
Moreover, $(u_{i},v_{i})\in C$ for $1\leq i\leq 3$. For this we have again three polynomials
\begin{align*}
    f_{4} &= (u_{1} - s)^{2} + (v_{1} - t)^{2} - r^{2}\\\nonumber
    f_{5} &= (u_{2} - s)^{2} + (v_{2} - t)^{2} - r^{2}\\\nonumber
    f_{6} &= (u_{3} - s)^{2} + (v_{3} - t)^{2} - r^{2}
\end{align*}
Finally,  $\nabla Q_{i}(u_{i},v_{i}) \wedge \nabla C(u_{i},v_{i})=0$ for $1\leq i\leq 3$, which is given by 
\begin{align*}
    f_{7} &= 2(u_{1}-s)(a_{2}u_{1} + 2a_{3}v_{1} + a_{5}) - 2(v_{1}-t)(2a_{1}u_{1} + a_{2}v_{1} + a_{4})\\\nonumber
    f_{8} &= 2(u_{2}-s)(b_{2}u_{2} + 2b_{3}v_{2} + b_{5}) - 2(v_{2}-t)(2b_{1}u_{2} + b_{2}v_{2} + b_{4}) \\\nonumber
    f_{9} &= 2(u_{3}-s)(c_{2}u_{3} + 2c_{3}v_{3} + c_{5}) - 2(v_{3}-t)(2c_{1}u_{3} + c_{2}v_{3} + c_{4})
\end{align*}
These $3$ types of constraints define a parametrized polynomial system of equations
\begin{equation}\label{polynomial_system}
F(x;p)=(f_{1},\dots,f_{9})^{T} = 0
\end{equation}
in the 9 variables $x=(u_{1}, v_{1}, u_2,v_2, u_{3},v_{3}, s, t, r)$ and  18 parameters given by the coefficients of each conic $p=(a_{1}, \ldots, a_{6}, b_{1}, \ldots, c_{6})$.
For a fixed set of parameters $p$ defining three conics, a solution $x \in \CC^9$ to the polynomial system $F(x;p)=0$ gives a circle with center $(s,t)$ and radius~$r$ that are tangent to $Q_1,Q_2$ and $Q_3$ at $(u_{1},v_{1}), \ldots, (u_{3},v_{3})$ respectively. 

\subsection{Complex circles tangent to three conics}
We begin by answering \Cref{ques:C_circles}. Observe that the B\'ezout bound of \eqref{polynomial_system} is $2^9 = 512$ which is strict in this case, as \cite{ApollonianCircles} shows\footnote{The authors show that there are at most $184$ circles tangent to three general ellipses. Since the space of ellipses is an open set in the space of conics, this bound applies to conics as well.} that there are at most $184$ circles tangent to three conics. We show this bound is attained for generic conics.

As in the case of Steiner's problem, the excess solutions arise in part from the locus of double lines. These double lines meet every conic at a point with multiplicity two, and are hence counted as tangent, regardless if the underlying reduced line is tangent or not.

Recall from \Cref{def_Q} that a conic is defined by 6 coefficients, so we can represent a conic by a point in $\mathbb P^5_{\mathbb C}$.
The locus of double lines is then precisely the image of the map from $\mathbb{P}^{2}_{\mathbb{C}}$, the space of lines, to $\mathbb{P}^{5}_{\mathbb{C}}$ by
$$[a:b:c]\mapsto[a^{2}:2ab:b^{2}:ac:bc:c^{2}].$$
This is the Veronese embedding, which we denote $V$. We can eliminate this excess intersection by blowing up our space of conics along $V$. Denote $X:=\mathrm{Bl}_{V}(\mathbb{P}^{5}_{\mathbb{C}})$ the blowup of $\mathbb{P}^{5}_{\mathbb{C}}$ along the Veronese surface $V$ and $\pi:\mathrm{Bl}_{V}(\mathbb{P}^{5}_{\mathbb{C}})\to\mathbb{P}^{5}_{\mathbb{C}}$ the blowing down morphism. The algebraic variety $X$ is called the \emph{space of complete conics}.

Fix a general point $p\in \mathbb{P}^{2}$, a general line $\ell \subset \mathbb{P}^{2}$ and a general conic $Q\subset \mathbb{P}^{2}$. Let $H_{p}, H_{\ell}, H_{Q}$ be the hypersurfaces in $\mathbb{P}^{5}_{\mathbb{C}}$ corresponding to conics that pass through~$p$, are tangent to $\ell$, or tangent to $Q$, respectively. We denote by $\widetilde{H_{p}}, \widetilde{H_{\ell}}, \widetilde{H_{Q}}$ the classes of $\pi^{-1}(H_p), \pi^{-1}(H_\ell), \pi^{-1}(H_Q)$ in the \emph{Chow ring} of $X$; see, e.g., \cite[Chapter 1]{3264AllThat} for more information on Chow rings and how they are used. Furthermore, we denote by $E$ the class of~$\pi^{-1}(V)=E\subseteq X$, called \emph{exceptional divisor}, away from which the blowing down map is an isomorphism of algebraic varieties. 

Recall that we want to count the number of conics that fulfill some intersection conditions. This corresponds to counting the number of points in an algebraic set of dimension zero that is defined by the intersection of  $\widetilde{H_{p}}, \widetilde{H_{\ell}}$ and $\widetilde{H_{Q}}$ in $X$. The classes $\widetilde{H_{p}}, \widetilde{H_{\ell}}, \widetilde{H_{Q}}$ do not intersect in the exceptional divisor $E$; see, e.g., \cite[p. 749-56]{PrinciplesAG}. This means that they meet transversely away from the subvariety of singular conics and have no common points in $E$. Thus any intersection problem involving conics, that contain a general point, or are tangent to a general line or a general conic, can be computed by taking the degree of the product of the corresponding classes in the Chow ring.

Let now $C\subseteq\mathbb{P}^{2}_{\mathbb{C}}$ be a circle. We can describe $C$ a circle with center $[s:t:r]$ and radius $r$ as the vanishing locus of the equation $(x-sz)^{2}+(y-tz)^{2}-r^{2}z^{2}\in\mathbb{C}[x,y,z]$. Note that $C$ passes through the circular points 
\begin{equation}\label{circular_points}
\circ_{+}:=[1:i:0]\quad \text{ and }\quad \circ_{-}:=[1:-i:0]
\end{equation}and conversely that any conic passing through $\circ_{+}, \circ_{-}$ is in fact a circle. Circles, then, are conics that pass through the circular points $\circ_{+},\circ_{-}$. We wish to enumerate the number of circles mutually tangent to three conics. By the above, these are precisely the conics mutually tangent to three general conics that also pass through the two circular points. The number of circles tangent to three general conics are thus given by the intersection product~$\widetilde{H_{Q}}^{3}\cdot\widetilde{H_{\circ_{+}}}\cdot\widetilde{H_{\circ_{-}}}$. 

Emiris and Tzoumas \cite{ApollonianCircles} used this observation and computed the upper bound $\widetilde{H_{Q}}^{3}\cdot\widetilde{H_{p}}
^2=184$, where $p$ is a general point in the sense of intersection theory (see also the survey article by Kleiman and Thorup \cite{TK97}). A priori, it is not clear that taking $p$ to be the circular points $\circ_{+}$ or $\circ_{-}$ is general in the sense of intersection theory. Therefore, for completeness we give the full proof of the fact that the answer to \Cref{ques:C_circles} is~184.

\begin{prop}\label{thm_main}
Given three general conics $Q_1,Q_2,Q_3 \subseteq \mathbb{C}^2$, there are 184 circles tangent to these three conics.
\end{prop}
\begin{proof}
We wish to compute $\widetilde{H_{Q}}^{3}\cdot\widetilde{H_{\circ_{+}}}\cdot\widetilde{H_{\circ_{-}}}$. We know that conics degenerate into flags, so the condition of being tangent to a conic $Q$ is equivalent to the condition that it contains either of two points or is tangent to either of two lines. This gives us the equality $\widetilde{H_{Q}}=2\widetilde{H_{p}}+2\widetilde{H_{\ell}}$; see also \cite[p.\ 775]{PrinciplesAG}.

We now seek to verify that 
$\widetilde{H_{p}}=\widetilde{H_{\circ_{+}}}, \widetilde{H_{\circ_{-}}}$. To show this, it suffices to show that the 4-planes in $\mathbb{P}^{5}_{\mathbb{C}}$ defined by conics passing through either of $\circ_{+},\circ_{-}$ do not contain the Veronese $V$ or, equivalently, that the equation defining the 4-plane vanishes to order zero on $V$; see \cite[p.\ 105]{CommAlg}. But since the hypersurfaces $H_{\circ_{+}}$ and $H_{\circ_{-}}$ do not contain the Veronese, we know that are $\widetilde{H_{\circ_{-}}}=\widetilde{H_{p}}+0\cdot E=\widetilde{H_{p}}$ and, similarly, $\widetilde{H_{\circ_{+}}} = \widetilde{H_{p}}$. Namely the circular points $\circ_{+}$ and $\circ_{-}$ are general in the sense of intersection theory.

Using the two facts above, we can enumerate the number of circles tangent to three general conics as the number of conics tangent to three general conics and passing through two general points. This gives us
\begin{align*}
    (2\widetilde{H_{p}}+2\widetilde{H_{\ell}})^{3}\cdot\widetilde{H_{\circ_{+}}}\cdot\widetilde{H_{\circ_{-}}} & = (2\widetilde{H_{p}}+2\widetilde{H_{\ell}})^{3}\cdot\widetilde{H_{p}}^{2} \\
    &= 8\cdot\widetilde{H_{p}}^{5} + 24\cdot\widetilde{H_{p}}^{4}\cdot\widetilde{H_{\ell}} + 24\cdot\widetilde{H_{p}}^{3}\cdot\widetilde{H_{\ell}}^{2} + 8\cdot\widetilde{H_{p}}^{2}\cdot\widetilde{H_{\ell}}^{3}. 
\end{align*}
Over $\mathbb{P}^{2}_{\mathbb{C}}$ there is one conic through 5 general points $\widetilde{H_{p}}^{5}=1$,  two conics through four general points and tangent to one general line $\widetilde{H_{p}}^{4}\cdot\widetilde{H_{\ell}}=2$, four conics through three general points and tangent to two general lines $\widetilde{H_{p}}^{3}\cdot\widetilde{H_{\ell}}^{2}=4$, and four conics through two general points and tangent to three general lines $\widetilde{H_{p}}^{2}\cdot\widetilde{H_{\ell}}^{3}=4$; see \cite[p.\ 307]{3264AllThat}. 
Making the appropriate substitutions yields 
$$8\cdot 1 + 24\cdot 2 + 24\cdot 4 + 8\cdot 4=184$$
which gives the claim. 
\end{proof}

\subsection{Real circles tangent to three conics} 
We now turn our attention to understanding to \Cref{ques:R_circles}. 

In the real version of Steiner's problem of finding the maximum number of real conics tangent to 5 general conics  \cite{3264RealConics}, the authors show that in a neighborhood of the corresponding real discriminant where all five conics are singular, there is an open cell where the number of real solutions achieves the complex upper bound. Using this same idea, we show that in our case this neighborhood produces at most $136$ real solutions. Notice that the following theorem, in particular, proves \Cref{thm:136_R}.

\begin{figure}
\begin{center}
\includegraphics[width = 0.9\textwidth]{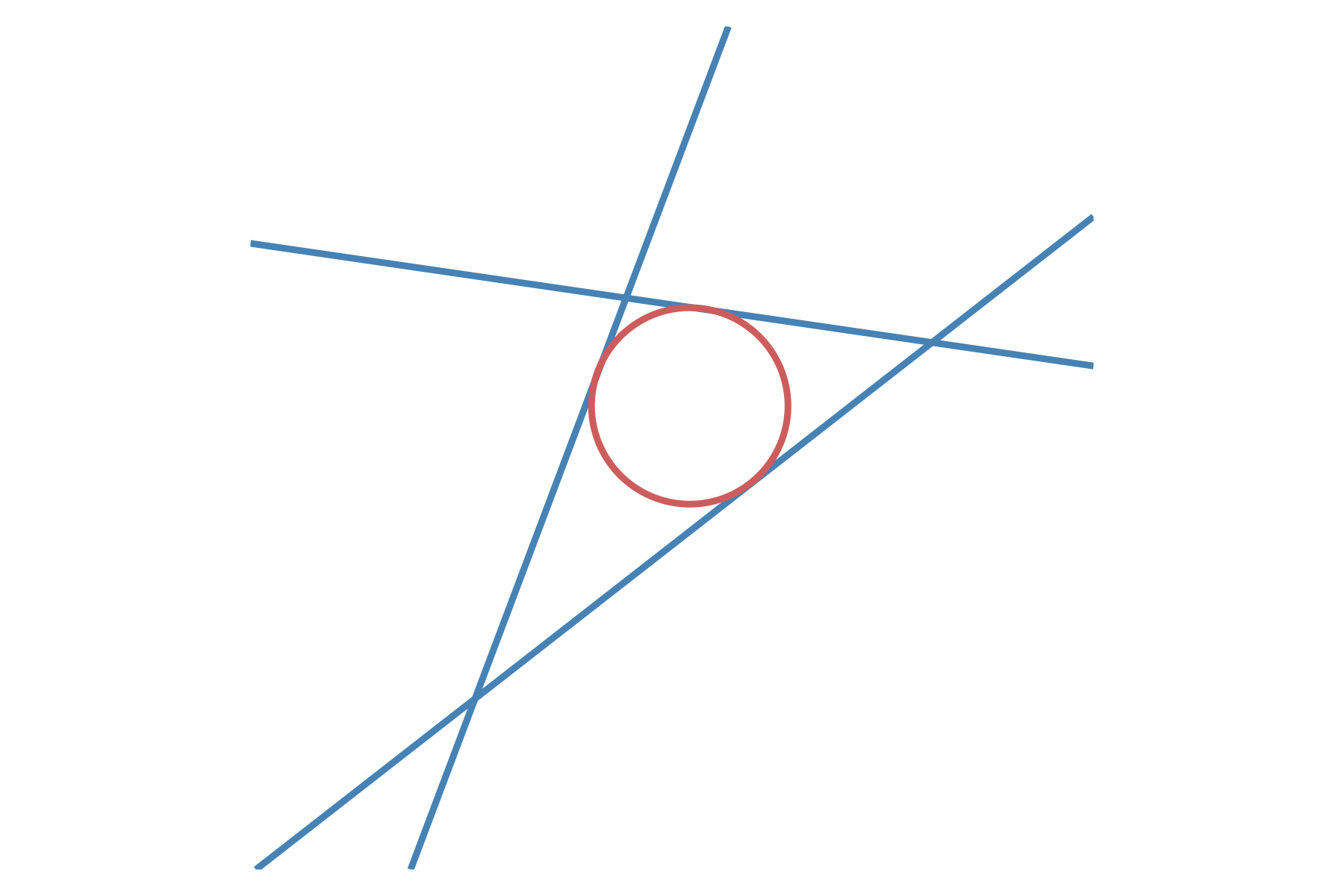}
\end{center}
\caption{\label{fig2}
A red circle tangent to three blue lines that form a triangle. The proof of \Cref{thm:136_R} considers such a triangle and finds a nearby arrangement of three hyperbolas, such that there are 136 real circles tangent to the three hyperbolas. In fact, the arrangement of hyperbolas in the proof of \Cref{thm:136_R_intro} is near the triangle shown in the picture.
}
\end{figure}

\begin{theorem}\label{thm:136_R}
Let $Q_1,Q_2,Q_3$ be three real conics in the plane such that $Q_1,Q_2,Q_3$ are all singular. There are at most $136$ real circles tangent to three conics in a neighborhood of $Q_1,Q_2,Q_3$. 
\end{theorem}
\begin{proof}
We adapt the argument of \cite{3264RealConics} as presented in \cite[Ch. 7]{RealSolutionsGeometry} of deforming a special configuration of conics. Suppose $\ell_{1},\ell_{2},\ell_{3}$ are lines supporting the edges of a triangle and $p_{i}\in\ell_{i}$ for $1\leq i\leq 3$ are points in the interior of the corresponding edge. 

We consider a subset of the lines $S \subseteq \{\ell_1,\ell_2,\ell_3 \}$ and a subset $P_S \subseteq \{p_1,p_2,p_3 \}$ of the points, such that for all $p_i \in P_S$, $\ell_i \not\in S$.
For every subset $S \subseteq \{ \ell_1, \ell_2, \ell_3 \}$ of the lines, there are 
$$\widetilde{H_{\circ_{+}}}\cdot\widetilde{H_{\circ_{-}}}\cdot\widetilde{H_{\ell}}^{|S|}\cdot\widetilde{H_{p}}^{3-|S|}=2^{\,\min\{|S|-2,\ 3-|S|\} + 2}$$
complex circles that are tangent to the lines in $S$ and meet the $3-|S|$ points in $P_S$. \Cref{fig2} shows an example of a circle that is tangent to three lines (i.e., $|S|=3$).

Note, however, for $|S|=2$, there are $4=2^2$ complex but only $2$ real circles tangent to two lines and passing through a point \cite{CirclesAndSpheres}. Altogether, this gives
$$\sum_{k=0}^{3}2^{\,\min\{k-2,\ 3-k\}+2}\,\binom{3}{k}=1\cdot\binom{3}{0}+2\cdot\binom{3}{1}+4\cdot\binom{3}{2}+4\cdot\binom{3}{3}=23$$
complex circles, but only 
$$1\cdot\binom{3}{0}+2\cdot\binom{3}{1}+2\cdot\binom{3}{2}+4\cdot\binom{3}{3}=17$$
real circles that for each $1\leq i\leq 3$ either meet $p_{i}$ or are tangent to $\ell_{i}$. 

With an asymmetric configuration, exactly $17$ of the real circles meet each point~$p_{i}$ and none of the $17$ tangent to $\ell_{i}$ are tangent at the point $p_{i}$. We now replace each pair $(p_{i},\ell_{i})$ with a smooth hyperbola $h_{i}$ that is asymptotically close to it: a hyperbola whose branches are close to $\ell_{i}$ and flex points close to $p_{i}$. 
If we do this for a pair $(p_{i},\ell_{i})$ then for every conic in our configuration there will be two nearby circles tangent to~$h_{i}$ -- one at each branch of the hyperbola. Replacing each pair $(p_{i},\ell_{i})$ by $h_i$ for $1\leq i\leq 3$, we get $2^{3}\cdot17=136$ real circles, proving our claim. 
\end{proof}

\Cref{thm:136_R} implies \Cref{thm:136_R_intro}. In the next section we give a constructive proof of~\Cref{thm:136_R_intro}.

\subsection{The real discriminant}
As a step towards the resolution of \Cref{conj_real} we characterize the real discriminant of the polynomial system \eqref{polynomial_system}, whose solutions describe circles tangent to three conics.

The real discriminant $\Delta \subseteq \RR^{18}$ of \eqref{polynomial_system} is a hypersurface in the space of real parameters where the parameters $p=(a_1,\ldots,a_6,b_1,\ldots,c_6) \in \Delta$ if and only if the number of real circles tangent to the three real conics defined by $p$ is not locally constant. 
We call such real parameters and the corresponding arrangement of conics \emph{degenerate}. In other words, $\Delta$ divides the parameter space  $\RR^{18}$ into open cells in which the number of real solutions to \eqref{polynomial_system} is constant. 

\begin{theorem}[The real discriminant]\label{thm:real_disc}
Let $\{ i,j,k \}  = \{1,2,3\}$.
An arrangement of three real conics $Q_1,Q_2,Q_3$ is degenerate, if and only if one the following holds:
\begin{enumerate}
\item There is a real line tangent to $Q_1,Q_2,Q_3$.
\item $Q_1,Q_2$ and $Q_3$ intersect in a real point.
\item $Q_i$ is singular at a real point $(u_i,v_i)$, and there is a real circle tangent to $Q_j,Q_k$ that passes through $(u_i,v_i)$.
\item $Q_i$ and $Q_j$ meet tangentially in a real point.
\item There exists a real circle $C$ that is tangent to $Q_1,Q_2,Q_3$ at the real points $(u_1,v_1), (u_2,v_2), (u_3,v_3)$, respectively, and the curvature of $C$ equals the curvature of $Q_i$ at~$(u_i,v_i)$ and the normal vectors $\nabla C(u_i,v_i)$ and $\nabla Q(u_i,v_i)$ point in the same direction.

\end{enumerate}
\end{theorem}
\begin{proof}

The discriminant consists of those real parameters $p\in\mathbb R^{18}$, where the polynomial system $F(x;p)=(f_1(x;p),\ldots,f_9(x;p))$ from \Cref{polynomial_system} has
\begin{enumerate}
\item a real solution at infinity.
\item a real solution $x\in \mathbb R^{9}$, such that the Jacobian matrix $$J_x=J_x(x;p)=\begin{bmatrix}\frac{\partial f_1(x;p)}{\partial x}& \ldots & \frac{\partial f_9(x;p)}{\partial x}\end{bmatrix}^T\in\mathbb R^{9\times 9}$$ of $F$ at $x$ is singular.
\end{enumerate} 

Let us first consider when a real solution goes off to infinity. First note that circles have only two points at infinity, namely $[ 1 : \pm i : 0 ]$. These are non-real points, so not limits of real solutions. Therefore, in order to have a real solution go to infinity, we must have at least one of $(a,b,r)$ go to infinity. If $r$ goes to infinity then it is a circle of infinite radius, which has curvature $0$ so it is a line. If $r$ is bounded then the circle converges to either $(x-a)^2 = 0$ or $(y-b)^2 = 0$ or $(x-a)^2 + (y-b)^2 = 0$. In any of these cases it is either a line or a point at infinity. Therefore, two real solutions go to infinity when there is a line tangent to all three conics, which proves the statement.

The rest of the proof consists of showing that cases (2)--(5) correspond exactly to those situation we the Jacobian matrix is singular.

If $r=0$, the three conics $Q_1,Q_2$ and~$Q_3$ must intersect in a point $(u,v)$ and we have the circle $(x-u)^2 + (y-v)^2 = 0$ tangent to all three conics, which gives a singular solution to the system (\ref{polynomial_system}). This is the second item above. Therefore, in the following we assume that $r\neq  0$.

The polynomial system in (\ref{polynomial_system}) consists of three triplets of polynomials, namely:
$$\psi_i(u_i,v_i,a,b,r):=\begin{pmatrix}
C(u_i,v_i)\\
Q_i(u_i,v_i)\\
\nabla C(u_i,v_i) \wedge
\nabla Q_i(u_i,v_i)
\end{pmatrix}\in\mathbb R^3 \text{ for } i = 1,2,3.
$$
Let $J_x^{(i)}$ denote the Jacobian matrix of $\psi_i$ at $x$. Then,
\begin{equation}\label{J_x_eq}
J_x = \begin{bmatrix} J_x^{(1)} \\ J_x^{(2)} \\ J_x^{(3)}\end{bmatrix} = \begin{bmatrix} A^{(1)} & 0 & 0 & B^{(1)} \\ 
    0 & A^{(2)} &  0 & B^{(2)} \\
    0 & 0 & A^{(3)} &  B^{(3)},
\end{bmatrix}
\end{equation}
where $A^{(i)}\in\mathbb R^{3\times 2}$ contains the partial derivatives of $\psi_i$ with respect to $(u_i,v_i)$, and $B^{(i)}\in\mathbb R^{3\times 3}$ contains the partial derivatives of $\psi_i$ with respect to $a,b,r$.

We have $\nabla C(x,y) = 2(x-a,y-b)^T$ and so
$$\frac{\partial\nabla C(u_i,v_i)}{\partial a} =
\begin{pmatrix}
-2\\0
\end{pmatrix},\quad 
\frac{\partial\nabla C(u_i,v_i)}{\partial b} =
\begin{pmatrix}
    0\\-2
    \end{pmatrix},\quad 
\frac{\partial\nabla C(u_i,v_i)}{\partial r} = \begin{pmatrix}
    0\\0
    \end{pmatrix}.$$
This shows that for $(\dot u_i,
\dot v_i,
\dot a,
\dot b,
\dot r)^T \in\mathbb R^5$ we have 
\begin{align}\label{jacobian_eq}
J_x^{(i)} \,\begin{pmatrix}
\dot u_i\\
\dot v_i\\
\dot a\\
\dot b\\
\dot r
\end{pmatrix} &= A^{(i)} \,\begin{pmatrix}
\dot u_i\\
\dot v_i
\end{pmatrix} + B^{(i)} \,\begin{pmatrix}
    \dot a\\
    \dot b\\
    \dot r
    \end{pmatrix}\\
    &=
\begin{pmatrix}
2\begin{pmatrix}\dot u_i \\ \dot v_i \end{pmatrix}^T \nabla C(u_i,v_i) + \dot C(u_i,v_i)\\[0.3em]
2\begin{pmatrix}\dot u_i \\ \dot v_i \end{pmatrix}^T \nabla Q(u_i,v_i)\\[0.3em]
2\begin{pmatrix}\dot u_i - \dot a\\ \dot v_i - \dot b\end{pmatrix} \wedge \nabla Q_i(u_i,v_i) + \nabla C (u_i,v_i) \wedge \mathrm H(Q_i)\begin{pmatrix}\dot u_i \\ \dot v_i \end{pmatrix} ,
\end{pmatrix}\nonumber
\end{align}
where $\dot C(x,y) = -2(x-a)\dot a-2(y- b)\dot b - 2r\dot r$ and 
$$\mathrm H(Q_i) = \begin{bmatrix}
\ \frac{\partial^2 Q_i}{\partial^2 u_i}\ & \ \frac{\partial^2 Q_i}{\partial u_i\partial v_i}\ \\[0.5em]
\ \frac{\partial^2 Q_i}{\partial u_i\partial v_i}\ & \ \frac{\partial^2 Q_i}{\partial^2 v_i}\
\end{bmatrix}$$
is the Hessian of $Q_i$ at $(u_i,v_i)$. 
Notice that $\dot C(x,y)$ is an affine linear function.

We have to show that the cases (3)--(5) above give exactly those situations, where we can find a nonzero vector $(\dot u_i,
\dot v_i,
\dot a,
\dot b,
\dot r)^T \in\mathbb R^5$ such that the vector in (\ref{jacobian_eq}) is equal to $0$ for each $i=1,2,3$. Since each of the equations is homogeneous in $u_i,v_i$, we can assume that
\begin{equation}\label{eq1}u_1^2 + v_1^2 = u_2^2 + v_2^2 =u_3^2 + v_3^2 =1.\end{equation}

If $Q_i$ is singular, we have a point $(u_i,v_i)$ with $Q_i(u_i,v_i)=0$ and $\nabla Q_i(u_i,v_i)=0$. The point $(u_i,v_i)$ is part of a solution $F(u_1,\ldots, v_3,a,b,r)=0$, if and only if there is a circle that is tangent to the other two conics $Q_j$ and $Q_k$ at $(u_j,v_j)$ and $(u_k,v_k)$, respectively, and that passes through $(u_i,v_i)$. 
For this data $J_x \,(
\dot u_i,
\dot v_i,
\dot a,
\dot b,
\dot r)^T=0$ becomes a system of $8$ linear equations in $9$ variables. This always has a nontrivial solution.

Next, since $C$ is tangent to $Q_i$ at $(u_i,v_i)$, we have $\nabla C(u_i,v_i)\wedge \nabla Q_i(u_i,v_i) = 0$; i.e, $\nabla C(u_i,v_i)$ is a multiple of~$\nabla Q(u_i,v_i)\neq 0$. The second entry in (\ref{jacobian_eq}) then implies 
$$\begin{pmatrix}\dot u_i \\ \dot v_i \end{pmatrix}^T \nabla C(u_i,v_i)=0,$$
so that~$\dot C(u_i,v_i)=0$ by the first entry. Unless $\dot C=0$, this implies that the three points $(u_1,v_1),(u_2,v_2),(u_3,v_3)$ lie on a line. Since they also lie on the circle of positive radius $r>0$, this implies $(u_i,v_i)=(u_j,v_j)$ for at least one pair $i\neq j$. But then $Q_i$ and~$Q_j$ intersect tangentially at $(u_i,v_i)$. In this case, in (\ref{J_x_eq}) we get $A^{(i)} = A^{(j)}$ and $B^{(i)} = B^{(j)}$, which gives a singular Jacobian $J_x$. This shows that item 4 above gives singular solutions.

So, outside the discriminant we must have $\dot C=0$; i.e., $\dot a=\dot b=\dot r=0$. Since~$\psi_i$ does not depend on $(u_j,v_j)$ for $j\neq i$, this means in order to understand when $J_x$ is singular, it is now enough to study when the equations in (\ref{jacobian_eq}) vanish. 
For $\dot C=0$ the third equation in (\ref{jacobian_eq}) becomes
\begin{equation}\label{jacobian_eq2}
2\begin{pmatrix}\dot u_i\\ \dot v_i \end{pmatrix} \wedge \nabla Q_i(u_i,v_i) -  \mathrm H(Q_i)\begin{pmatrix}\dot u_i \\ \dot v_i \end{pmatrix}\wedge  \nabla C (u_i,v_i)=0.
\end{equation}
We have $(\nabla C (u_i,v_i))^T \nabla C (u_i,v_i) = 4r^2>0$ and $(\nabla Q (u_i,v_i))^T \nabla Q (u_i,v_i)>0$, since~$Q_i$ is smooth. Then, multiplying (\ref{jacobian_eq2}) by $2r\sqrt{(\nabla Q (u_i,v_i))^T \nabla Q (u_i,v_i)}$ and using (\ref{eq1}) and the fact that
$$\begin{pmatrix}\dot u_i \\ \dot v_i \end{pmatrix}^T \nabla Q(u_i,v_i) = \begin{pmatrix}\dot u_i \\ \dot v_i \end{pmatrix}^T \nabla C(u_i,v_i)=0$$
we have
$$\frac{\varepsilon_i}{r} = 
\begin{pmatrix}\dot u_i \\ \dot v_i \end{pmatrix}^T
\frac{\mathrm H(Q_i)}{\sqrt{(\nabla Q (u_i,v_i))^T \nabla Q (u_i,v_i)}}
\begin{pmatrix}\dot u_i \\ \dot v_i \end{pmatrix},$$
where $\varepsilon_i = 1$, if $\nabla Q(u_i,v_i)$ and $\nabla Q(u_i,v_i)$ point into the same direction, and $\varepsilon_i = -1$ otherwise. Since, $r^{-1}$ is the curvature of $C$ and the right hand side equals the curvature of~$Q_i$ at $(u_i,v_i)$, this shows that we have a singular solution also in the fifth item above. In all other cases, $J_x$ has a trivial kernel, hence is not singular. 
\end{proof}

\section{Hill climbing}\label{sec:hillClimbing}
\Cref{thm:136_R} shows that there exist three real conics that have 136 real circles tangent to all three, but it does not provide an explicit construction of the three conics. To find conics that exhibit this behavior, we rely on a numerical method known as \emph{hill climbing}. We adapt a method of Dietmaier in \cite{LinearProgram} to increase the count of real solutions.\footnote{Dietmaier's hill climbing algorithm was recently applied in \cite{TangentQuadrics} for generating instances of points, lines, and surfaces in 3-space with a maximal number of real quadrics, that contain the points and are tangent to the lines and surfaces.} For the ease of exposition, we outline our method below using matrix inverses. In our implementation we do not invert any matrices and instead we introduce auxiliary variables and solve an equivalent linear system, allowing for more numerically stable computations.

The basic idea of the hill climbing algorithm is as follows. Suppose we are given a set of parameters $p=(a_{1},\dots,a_{6},b_{1},\dots, c_{6})\in\mathbb{R}^{18}$ defining a configuration of three general conics in the plane. By \Cref{thm_main}, the system of polynomial equations $F(x;p)=0$ from \eqref{polynomial_system} has 184 complex solutions. To increase the number of real solutions, we iteratively perturb $p$ so that a complex conjugate solution pair first becomes a double real root then perturb $p$ once again to separate this double root resulting in two distinct real roots. Simultaneously, we ensure that no existing real solution vectors become arbitrarily close, forming a double root and eventually a complex conjugate pair and that solutions do not diverge to infinity. 

For fixed parameters $p \in \RR^{18}$, denote $S$ by
$$S=S_{\mathbb{R}}\sqcup S_{\mathbb{C}}\subseteq\mathbb{C}^{9}$$
the solutions of our system of polynomial equations \eqref{polynomial_system} where $S_{\mathbb{R}}$ is the set of solutions with only real entries and $S_{\mathbb{C}} = S\setminus S_{\mathbb{R}}$.

In the first step of the hill climbing algorithm, we select one solution $x^{*}\in S_{\mathbb{C}}$ in which we aim to decrease the $L^1$ norm $\Vert\mathrm{Im}(x^{*})\Vert_{1}$ of the vector of imaginary parts of $x^*$. The goal is to compute a step $\Delta p$ in the parameter space $\mathbb{R}^{18}$, such that the magnitude of the imaginary part of $x^{*}$ decreases as we move from $p$ to $p+\Delta p$ for some $-\varepsilon\mathbbm{1}\leq\Delta p\leq \varepsilon\mathbbm{1}$ where $\varepsilon$ is a small tolerance parameter and $\mathbbm{1}=(1,\ldots,1)^T\in\mathbb R^{18}$ is the all-one vector.

As in the previous section we denote by $J_{x}(x;p)\in\mathbb{C}^{9\times 9}$ the Jacobian matrix of $F(x;p)$ with respect to the variables $x=(u_{1},v_{1}, \dots, s,t,r)$ evaluated at $x$ with parameters $p$.
Similarly, $J_{p}(x;p)\in\mathbb{C}^{9\times 18}$ is the Jacobian of $F(x;p)$ with respect to the parameters $p=(a_{1},\dots, a_{6},b_{1},\dots,b_{6})$ evaluated at $x$ with parameters $p$. 
Following \cite{LinearProgram} we observe that differentiating both sides of $F(x;p)=0$ with respect to $p$ and $x$, gives the following matrix equation involving the step $\Delta p$:
\begin{equation}
    J_{x}(x;p)\Delta x + J_{p}(x;p)\Delta p=0. \label{matrixeqn}
\end{equation}
Solving \eqref{matrixeqn} for $\Delta x$, we have that 
\begin{align} \text{Im}(\Delta x) = -\text{Im}(J_x(x;p)^{-1}\cdot J_p(x;p) ) \cdot \Delta p, \label{eq:delta_x}
\end{align}
where $ \text{Im}(\cdot)$ denotes taking the componentwise imaginary part. Let $ \text{sign}(\cdot)$ denote the componentwise sign function.
In order to decrease $\lVert \Delta \text{Im}(x^*) \rVert_1$, we wish to minimize
\[ -\text{sign}(\text{Im}(x^*))^T \cdot \text{Im}(J_x(x^*;p)^{-1}\cdot J_p(x^*;p) ) \cdot \Delta p.\]
Notice that this objective function considers a first order approximation of our system $F(x,p)$ at $(x^*,p)$, so we add the constraint $-\varepsilon\mathbbm{1}\leq\Delta p\leq \varepsilon\mathbbm{1}$ to ensure that this approximation is accurate enough.

Next, we want to ensure that as we take a step in the parameter space, two existing real solutions do not come together to become non-real. Consider two real solutions $x_i, x_j \in S_{\RR}$. The distance between $x_i, x_j$ is the $L^2$ norm $D = \lVert x_i - x_j \rVert_2^2$. Differentiating $D$ with respect to $x$ yields
\[ 2 \cdot \langle x_i - x_j, \Delta x_i - \Delta x_j \rangle .\]
Substituting \eqref{eq:delta_x} in for $\Delta x_i, \Delta x_j$, we have an expression that gives the change in the distance between two real solutions $x_i,x_j$ as a function of the change in parameters. Since we do not want this distance to decrease, we impose the constraint
\[\forall x_i, x_j \in S_{\RR}: (x_i - x_j)^T \cdot \big( J_x(x_i;p)^{-1} \cdot J_p(x_i;p) - J_x(x_j;p)^{-1} \cdot J_p(x_j;p) \big) \cdot \Delta p \geq 0.\]

Finally, we want to ensure that a complex solution does not go off to infinity as we take a step in the direction $\Delta p$. To enforce this constraint, we consider the magnitude of every complex solution $x_i \in S_{\CC}$ and impose that the magnitude does not increase. Define the following $18\times 18$ block matrices
\begin{align*}
\widetilde{J_{x}}(x;p) &= 
\begin{bmatrix}
\ \text{Re}(J_{x}(x;p)) & 0\  \\ \ 0 & \text{Im}(J_{x}(x;p))\ \end{bmatrix}\in\mathbb{R}^{18\times18},\\[0.5em]
\widetilde{J_{p}}(x;p) &= 
\begin{bmatrix}
\ \text{Re}(J_{p}(x;p))\   \\ \  \text{Im}(J_{p}(x;p))\ 
\end{bmatrix}\in\mathbb{R}^{18\times18}
\end{align*} 
and consider the augmented vector
\[ \Tilde{x}_i = \begin{bmatrix}
\text{Re}(x_i) \\ \text{Im}(x_i)
\end{bmatrix}\in\mathbb R^{18}. \]
The magnitude of $x_i$ is the same as $\lVert \Tilde{x}_i \rVert_2^2$. Again, we differentiate $\lVert \Tilde{x}_i \rVert_2^2$ and use \eqref{eq:delta_x} (considering $\text{Re}(x_i)$ and $\text{Im}(x_i)$ as separate elements) to write
\[ \forall x_i \in S_{\CC} : \langle \Tilde{x}_i \cdot \Tilde{J}_x(x_i;p)^{-1} \cdot \Tilde{J}_p(x_i;p), \Delta p \rangle \geq 0. \]
This constraint ensures that the change in the magnitude of $x_i$ does not increase.

In summary, in the first step of our hill climbing algorithm we consider the linear program:
\footnotesize
\begin{align*}
    & \hspace{0.2cm} \min_{\Delta p} \  -\text{sign}(\text{Im}(x^*))^T \cdot \text{Im}(J_x(x^*;p)^{-1}\cdot J_p(x^*;p) ) \cdot \Delta p  \label{opt-C} \tag{Opt-$\CC$} \\
    \text{subject to} &\hspace{0.2cm} - \epsilon \cdot \mathbbm{1} \leq \Delta p \leq \epsilon \cdot \mathbbm{1} \\
    &\hspace{0.2cm} \forall \ x_i, x_j \in S_{\RR}: (x_i - x_j)^T \cdot \big( J_x(x_i;p)^{-1} \cdot J_p(x_i;p) - J_x(x_j;p)^{-1} \cdot J_p(x_j;p) \big) \cdot \Delta p \geq 0\\
    &\hspace{0.2cm} \forall x_i \in S_{\CC}: \langle \Tilde{x}_i \cdot \Tilde{J}_x(x_i;p)^{-1} \cdot \Tilde{J}_p(x_i;p), \Delta p \rangle \geq 0.
\end{align*}
\normalsize
So long as $\varepsilon$ is sufficiently small so that the first order approximation of $F(x;p)$ is accurate, an optimal solution to \ref{opt-C}, $\Delta p^{*}$, gives a step in the parameter space in which the magnitude of the imaginary part of $x^{*}$ decreases.

Algorithm~\ref{alg:new_p} repeatedly solves \eqref{opt-C} and updates $x^*$ until $\lVert \text{Im}(x^*) \rVert_{1}$ is sufficiently small. At this point, $x^*$ is close to being a singular real root.
 
\smallskip
\begin{algorithm}[hbt!]
\DontPrintSemicolon 
\KwIn{Parameters $p \in \mathbb{R}^{18}$ and a non-real solution $x^* \in \mathbb{C}^9$ such that $F(x;p) = 0$ where $F$ is as defined in \eqref{polynomial_system} and a tolerance $\epsilon >0$
}
\KwOut{
    Parameters $p' \in \RR^{18}$ and a non-real solution $x' \in \CC^9$ such that $F(x',p') = 0$ where $F$ is as defined in \eqref{polynomial_system} and $\lVert \text{Im}(x') \rVert_2 \leq \epsilon$
}
  \While{  $\lVert \mathrm{Im}(x^*) \rVert_2 > \epsilon$}
  {Solve the optimization problem \eqref{opt-C} for $\Delta p^*$ \;
    Set $p' = p + \Delta p^*$ \;
   Using parameter continuation, compute $S = \{x' \in \CC^9 \mid  F(p',x') = 0 \}$ 
    \; 
    Set $x^*:= \arg\min_{x \in S} \{ \lVert x - x^* \rVert_2 \} $     \; }
        {Return $ p',x' := x^*$}
\caption{Minimize imaginary norm}\label{alg:new_p}
\end{algorithm}
\smallskip

Given an (almost) singular real root, $x^{*}$ satisfying $F(x^{*};p)=0$, we separate it into two real roots by first setting the imaginary part of $x^*$ equal to zero and then adding a small quantity $\delta \in \RR^9$ component-wise yielding
$$x_{1}'=x^{*}+\delta, \hspace{0.5cm} x_{2}'=x^{*}-\delta\in\mathbb{R}^{9}.$$
We then find conics $\hat{p}\in\mathbb{R}^{18}$ close to $p$ that contain $x_1',x_2'$ as points of tangency by solving the following optimization problem:
\smallskip

\small
\begin{align*}
    \tag{Opt-$p$}\label{opt-p} &\hspace{0.2cm} \arg\min_{\hat{p}\in\mathbb{R}^{18}}\Vert \hat{p}-p\Vert^{2} \\
    \text{subject to} &\hspace{0.2cm}  f_{1}(x_{1}';\hat{p})=f_{1}(x_{2}';\hat{p})=0 \\
    &\hspace{0.2cm}  f_{2}(x_{1}';\hat{p})=f_{2}(x_{2}';\hat{p})=0 \\
    &\hspace{0.2cm}  f_{3}(x_{1}';\hat{p})=f_{3}(x_{2}';\hat{p})=0 \\
    &\hspace{0.2cm}  f_{7}(x_{1}';\hat{p})=f_{7}(x_{2}';\hat{p})=0 \\
    &\hspace{0.2cm}  f_{8}(x_{1}';\hat{p})=f_{8}(x_{2}';\hat{p})=0 \\
    &\hspace{0.2cm}  f_{9}(x_{1}';\hat{p})=f_{9}(x_{2}';\hat{p})=0
\end{align*}
\normalsize

Observe that now that $F(x_{1}';p)\neq0$ and $F(x_{2}';p)\neq0$, because we did not use $f_4,f_5,f_6$ for the constraints. In fact, it might not even be possible to find a parameter $\hat{p}$ such that $F(x_1',\hat p)=F(x_2',\hat p)=0$. For instance, if $x_1'=(u_1,\ldots,v_3,a,b,r)$, then there is no reason to expect that $(u_i,v_i)$ are on the circle defined by $a,b,r$. Nevertheless, computing $\hat p$ in in \eqref{opt-p} we find conics close to our original conics. We use parameter homotopy continuation (see, e.g., \cite[Section 7]{Sommese:Wampler:2005}) to find all $x$ such that $F(x; \hat{p}) = 0$ and select $x_1,x_2$ closest to $x_1',x_2'$.

While $x_{1}$ and $x_{2}$ are two distinct real roots, they are still close together, meaning they are close to being a complex conjugate pair. Therefore we would like to separate them so that they are further apart. 

Recall, that we can express the change in distance between two real solutions, $x_1,x_2$ as a function of the change in parameters $\Delta p$ by
\[ (x_1 - x_2)^T \cdot \big( J_x(x_1;p)^{-1} \cdot J_p(x_1;p) - J_x(x_2;p)^{-1} \cdot J_p(x_2;p) \big) \cdot \Delta p. \]
We would like to maximize this function still subject to the constraints above that $\Delta p$ is constrained to a small neighborhood and that none of the other real solutions become too close and none of the other complex solutions become too large. This is equivalent to solving:

\footnotesize
\begin{align*}
    \label{opt-R}\tag{Opt-$\RR$} &\hspace{0.2cm} \max_{\Delta p} \  (x_1 - x_2)^T \cdot \big( J_x(x_1;p)^{-1} \cdot J_p(x_1;p) - J_x(x_2;p)^{-1} \cdot J_p(x_2;p) \big) \cdot \Delta p  \\
    \text{subject to} & \hspace{0.2cm}- \epsilon \cdot \mathbbm{1} \leq \Delta p \leq \epsilon \cdot \mathbbm{1} \\
    &\hspace{0.2cm} \forall \ x_i, x_j \in S_{\RR}: (x_i - x_j)^T \cdot \big( J_x(x_i;p)^{-1} \cdot J_p(x_i;p) - J_x(x_j;p)^{-1} \cdot J_p(x_j;p) \big) \cdot \Delta p \geq 0  \\
    &\hspace{0.2cm} \forall x_i \in S_{\CC}: \langle \Tilde{x}_i \cdot \Tilde{J}_x(x_i;p)^{-1} \cdot \Tilde{J}_p(x_i;p), \Delta p \rangle \geq 0
\end{align*}
\normalsize

Combining these three optimization problems defines our hill climbing algorithm. It successively finds parameter values with higher numbers of real roots. We first repeatedly solve \ref{opt-C} to make the imaginary part of a given root sufficiently small resulting in a singular real root. We then use \ref{opt-p} to find a set of parameters that matches a small perturbation of our singular root, before applying \ref{opt-R} to separate then as real roots. This procedure is outlined in \Cref{alg:hill_climb}.
\smallskip

\begin{algorithm}[hbt!]
\DontPrintSemicolon 
\KwIn{Parameters $p \in \mathbb{R}^{18}$ and real and complex solutions $S_{\CC}$ and $S_{\RR}$ such that for all $x \in S_{\CC}$ and $x \in S_{\RR}$, $F(p,x) = 0$ and a tolerance $\epsilon$
}
\KwOut{
    Parameters $p' \in \RR^{18}$ where the number of non-real solutions to $F(p',x) = 0$ is strictly less than $|S_{\CC}|$
}
 Select $x^* \in S_{\CC}$ and run \Cref{alg:new_p} to output $p',x'$ \;
  Solve \eqref{opt-p} to obtain output $\hat{p}$ \;
   Solve \eqref{opt-R} \; 
      {Return $ p'$}
\caption{Hill climbing}\label{alg:hill_climb}
\end{algorithm}
\smallskip

We implement \Cref{alg:hill_climb} in \texttt{Julia} and provide the necessary code and documentation on our \texttt{MathRepo} page 

We can now explain how we prove \Cref{thm:136_R_intro}.
\begin{proof}[Proof of Theorem \ref{thm:136_R_intro}.]
We first generate a parameter $q\in\mathbb R^{18}$ that defines a triangle (three degenerate conics) as in the proof of \Cref{thm:136_R}. We know from the proof that in the neighborhood of $q$ there must be a parameter with $136$ real circles. So, we add a small random perturbation to $q$ and obtain a parameter $p$. This parameter is then used as the starting point for the hill climbing algorithm \Cref{alg:hill_climb}. Eventually, we get the following three conics: 
\smallskip

{\footnotesize
\begin{align*}
    Q_{1}& = \left(\frac{400141104595769}{2302676434480590430}\right) x^{2} + \left(\frac{5537854491843451}{2305843009213693952}\right) xy + \left(\frac{2379998783885947}{288230376151711744}\right) y^{2}\\
    & \hspace{0.7cm} - \left(\frac{ 5883336424977557}{288230376151711744}\right) x - \left(\frac{5057485722682341}{36028797018963968}\right) y + \left(\frac{2686777020175459}{4503599627370496}\right) \\[1em]
    Q_{2} &= \left(\frac{2326975324861901}{144115188075855872}\right)x^2 - \left(\frac{7017759077361941}{576460752303423488}\right)xy + \left(\frac{5286233514864229}{2305843009213693952}\right)y^{2} \\
    & \hspace{0.7cm} + \left(\frac{3536130883475143}{18014398509481984}\right) x - \left(\frac{5331739727004679}{72057594037927936}\right) y + \left(\frac{5373554039379455}{9007199254740992}\right) \\[1em]
    Q_{3} &= \left(\frac{6288284117996449}{576460752303423488}\right) x^{2} - \left(\frac{8069853070614251}{288230376151711744}\right)xy + \left(\frac{1293970525023733}{72057594037927936}\right)y^{2}\\
    & \hspace{0.7cm} - \left(\frac{1453444402131837}{9007199254740992}\right)x + \left(\frac{7458321785480773}{36028797018963968}\right) y + \left(\frac{2686777019781135}{4503599627370496}\right)
\end{align*}}
\smallskip

Using the software \texttt{HomotopyContinuation.jl} \cite{HomotopyContinuation} we solve the system of polynomial equations \Cref{polynomial_system} and get 136 real solutions (in floating point arithmetic). These 136 numerical solutions are then certified by interval arithmetic \cite{Certifying,HomotopyContinuation}. This gives a proof that the three conics above have indeed 136 tritangent real circles.
\end{proof}
A \texttt{Julia} file for certification of the above polynomial system is available on our \texttt{MathRepo} page. While \Cref{alg:hill_climb} never found an instance of conics with more than $136$ tritangent real circles and \Cref{thm:136_R} shows that similar arguments in~\cite{3264RealConics} cannot be used to show that there exist three conics with $184$ tritangent real circles, it does not exclude the possibility that such conics do exist. That being said, we conjecture that the maximum number of real circles tangent to three general, real conics is $136$ and we interpret \Cref{thm:136_R} and our numerical experiments running \Cref{alg:hill_climb} as strong evidence supporting this conjecture.

With the help of \Cref{alg:hill_climb} we also find 69 distinct parameters $p_0,\ldots,p_{68}\in\mathbb R^{18}$ such that the number of real circles corresponding to the conic arrangement defined by $p_k$ is $2k$. That is, for every even number $2k$ between 0 and 136 we find a parameter that gives $2k$ real circles. Using certification by interval arithmetic \cite{Certifying,HomotopyContinuation} we then have a proof for the next theorem.
\begin{theorem}\label{thm:all are possible}
Let $0\leq n\leq 136$ be an even number. Then, there exists a parameter $p\in\mathbb R^{18}$, which is outside the real discriminant, such that the conic arrangement corresponding to $p$ has exactly $n$ real circles that are tangent to these conics.
\end{theorem}
The data that proves this theorem can be found on our \texttt{MathRepo} page.

\begin{figure}
\begin{center}
\includegraphics[width = 0.95\textwidth]{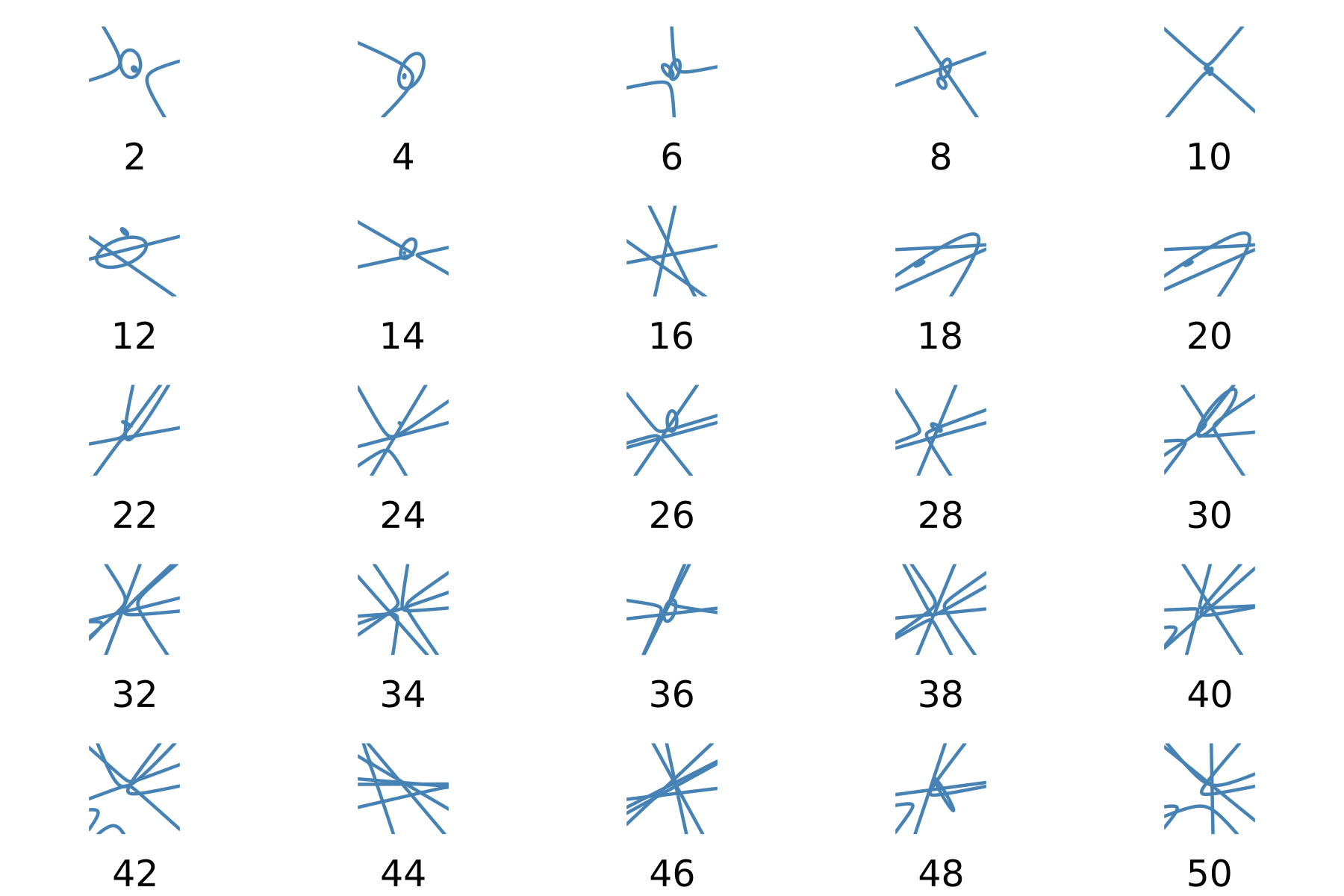}

\includegraphics[width = 0.95\textwidth]{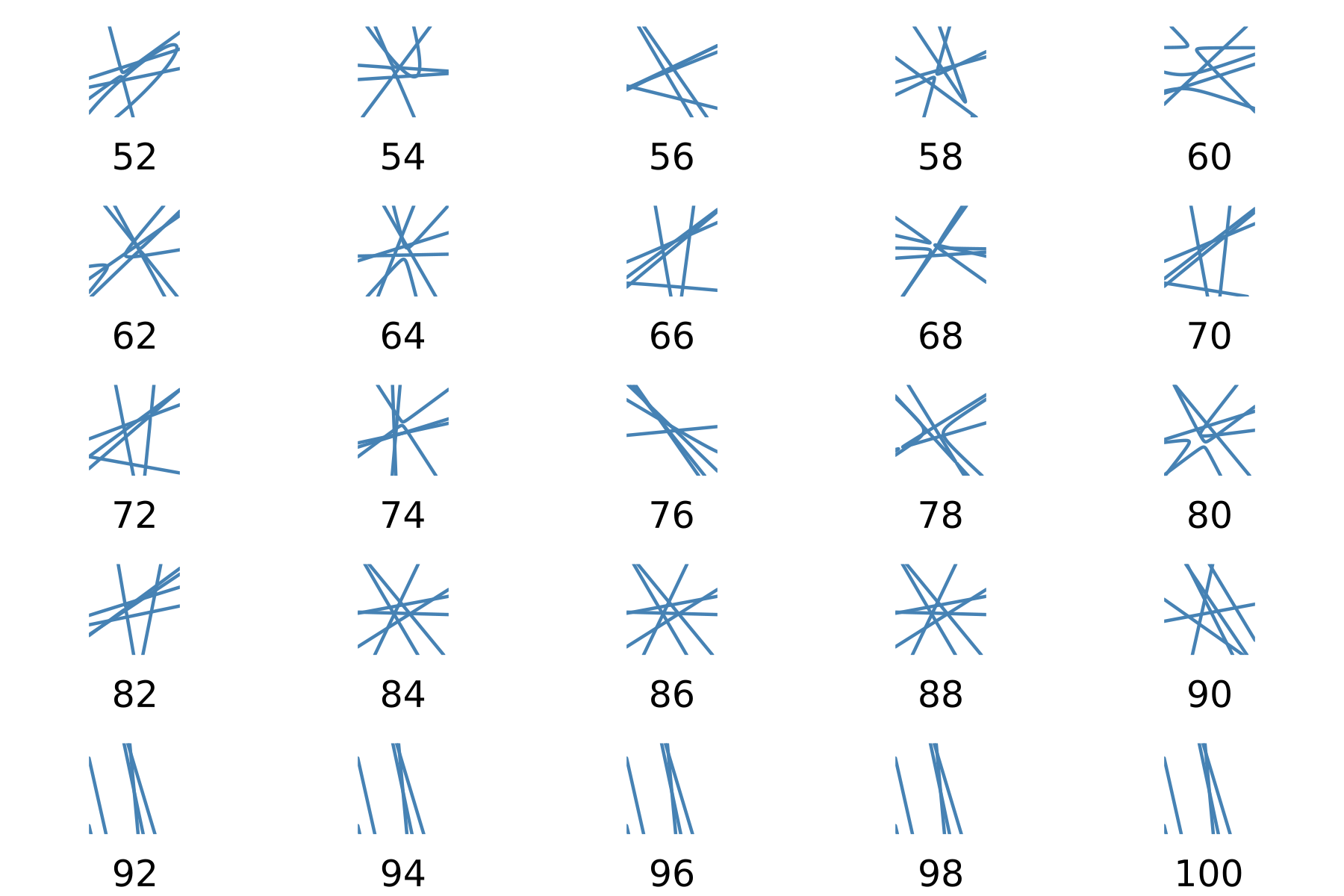}
\end{center}
\caption{\label{fig4} We show the first 50 conic arrangements that were used to prove \Cref{thm:all are possible}. The number below each plot indicates the number of real tritangent circles. Starting from 92 real circles the arrangement all look similar to a triangle as in \Cref{fig2}.}
\end{figure}

\section{Machine learning}\label{sec:ml}
In this section we investigate to what extent machine learning algorithms are able to input a real parameter vector $p\in\mathbb R^{18}$ and predict the number of real circles tangent to the three conics $Q_1,Q_2,Q_3$ defined by $p$. We use a supervised learning framework. This means our data conists of points
$(p,n)\in \mathbb R^{18} \times \{0,2,4,\ldots,184\},$
where $n$ gives the number of real circles corresponding to $p$ (by \Cref{conj_real}, we believe that $\{0,2,4,\ldots,136\}$ suffices). In the language of machine learning $p$ is called the \emph{input data} and $n$ is called the \emph{label} or \emph{output-variable}.

We first describe in the next section how we generate our data. Then, in \Cref{sec:our model} we explain our machine learning model and in \Cref{sec:evaluation} we discuss how well it performs.

\subsection{Data generation and encoding}\label{sec:data}

We consider two training data sets, $\mathcal{D}_1$ and $\mathcal{D}_2$.
To generate $\mathcal{D}_1$, we first sample parameters, $p \in \mathbb{R}^{18}$ defining the three conics $Q_1,Q_2,Q_3$ from a normal $\mathcal{N}(0_{18},I_{18})$ distribution where $0_{18} \in \RR^{18}$ is the all zeroes vector and $I_{18} \in \RR^{18 \times 18}$ is the identity matrix. We compute the number of real zeros $n$  using the software \texttt{HomotopyContinuation.jl} \cite{HomotopyContinuation}, and then we perform the hill climbing algorithm outlined in \Cref{alg:hill_climb}. Hill climbing is necessary in order to have samples with high numbers of real solutions.  To generate $\mathcal{D}_1$, we execute \Cref{alg:generate} for $M=50,000$.

\smallskip
\begin{algorithm}[hbt!]
\DontPrintSemicolon 
\KwIn{A number $M$}
\KwOut{A data set $\mathcal D_1$ with $\vert\mathcal D_1\vert = M$
}
\While{$\vert\mathcal D_1\vert<M$}{
Randomly select $p \in \RR^{18}$ from a normal $\mathcal{N}(0_{18},I_{18})$ distribution \;
Compute the number of real circles, $n$, tangent to conics $Q_1, Q_2$ and $Q_3$ whose coefficients are defined by $p$\;
Add $(p,y)$ to $\mathcal{D}_1$\;
Run \Cref{alg:hill_climb} with input $p$ to get output $p'$ and compute the number of real circles, $n'$, tangent to conics $Q_1',Q_2',Q_3'$ defined by $p'$\;
\uIf{$n'>n$}{
    add $(p',n')$ to $\mathcal{D}_1$ and repeat step $4$ with input $p'$\;
  }
\Else{
    Go back to Step $1$\;
}
}
\caption{Generate $\mathcal D_1$}\label{alg:generate} 
\end{algorithm}
\smallskip

To generate the training set $\mathcal{D}_2$, we sample $50,000$ data points independently from a normal $\mathcal{N}(0_{18},I_{18})$ distribution and find the number of real circles tangent to the corresponding conics using the software \texttt{HomotopyContinuation.jl}. 

The data sets $\mathcal{D}_1,\mathcal{D}_2$ are plotted in histograms in \Cref{fig_histo}. As one can see the purely random data $\mathcal D_2$ has a large number of instances with no real tritangent circles and concentrates around $20$ real tritangent circles with fast decay. In addition, $\mathcal{D}_2$ does not represent any arrangements of conics with more than 60 real circles. By contrast the data $\mathcal D_1$ is much more representative of arrangements with many real circles. This can be explained by the results in \cite{BKL2022, DL2022, LS2021} where the authors show that polynomials with Gaussian coefficients tend to have a ``simple'' topology with high probability. In our case, \cite{BKL2022, DL2022, LS2021} imply that the probability of having many real circles is exponentially small. This phenomenon can be seen in \Cref{fig_histo}.

\begin{figure}[!tbp]
    \includegraphics[width=0.7\textwidth]{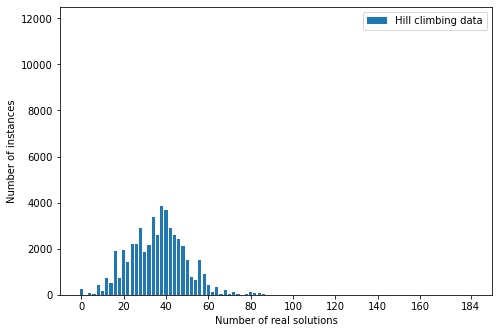}

    \includegraphics[width=0.7\textwidth]{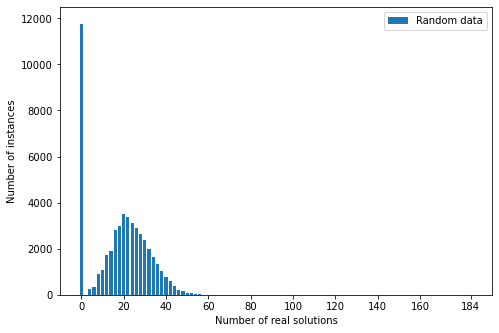}
    
  \caption{\label{fig_histo}Two histograms showing the distribution of the number of real circles tangent to a given configuration of conics. The top histogram shows the distribution on the data $\mathcal D_1$ we generated using the hill climbing algorithm and the bottom histogram shows the distribution on random data $\mathcal D_2$.}
\end{figure}

\subsection{The model} \label{sec:our model}
Since the number of real circles $n$ tritangent to $3$ conics is a discrete variable, our problem is a \emph{classification problem} -- the data space for the output variable consists of discrete points, called \emph{classes}. 

We found that turning our problem into a regression problem works better: Instead of having a response variable in $n \in \mathbb N$, we consider a \emph{statistical model}
$$(p,y)\in \mathbb R^{18} \times \Delta_{92},$$
where $\Delta_{92} = \{y\in\mathbb R^{93} \mid y_1+\cdots+y_{93}=1, y_i\geq 0\}$ is the $92$-dimensional standard simplex. The underlying idea is that $y$ defines a discrete random variable where $y_k:=\mathrm{Prob}(y=k)$ gives the probability that for given input parameters $p \in \RR^{18}$, the corresponding conics defined by $p$ have $2k$ real tritangent circles. For all parameters outside of the discriminant, there must be an even number of real solutions, so we have $1+\tfrac{184}{2}= 93$ possibilities for the number of real zeros. The data points from the previous subsection are then encoded as the vertices of the simples; i.e., given a number of real circles $n$ we associate to it the probability distribution $y$, where $\mathrm{Prob}(y=\tfrac{1}{2}n)=1$. This process is sometimes called \emph{one hot encoding}. It turns categorical data into data which can be used for regression problems. 

The goal of our machine algorithm is then to learn a function 
$$\phi: \mathbb R^{18} \to \Delta_{92}.$$
such that $\phi(p)$ is a good predictor for the number of real circles corresponding to $p$. Since $\phi(p)$ is a point in $\Delta_{92}$, we predict the number of real tritangent circles to the three conics defined by $p$ to be two times the maximum index of $\phi(p)$:
\[ 2 \cdot \arg\max \ \{ \phi(p)_{i-1} \ : \ 1 \leq i \leq 93\} \in \{0,2,\ldots,184\}.\]

We model $\phi$ using a \emph{multilayer perceptron} (MLP) \cite{HORNIK1991251}. A MLP is defined as the composition $\phi = \phi_1\circ\cdots\circ\phi_m$ of sub-functions $\phi_0,\ldots,\phi_m$ called \emph{layers}, where 
$$\phi_i(x) = \sigma_i(W_ix+b_i)$$
with matrices $W_{i}\in\mathbb R^{m_{i}\times m_{i+1}}$, vectors $b_i\in\mathbb R^{m_i}$, and a nonlinear \emph{activation function} $\sigma_i: \mathbb R^{m_i}\to\mathbb R^{m_{i}}$. The matrices $W_i$ are called \emph{weights} and the vectors $b_i$ are called \emph{biases}. In our model we use $m=2$ layers, and do not impose any sparsity constraint on the weights. That is, we use a \emph{fully connected model}. Both layers have $1,000$ neurons each, which means that $m_1=m_2=1,000$. As the nonlinear activation function we use the coordinate-wise \emph{ReLU function}  
$$\sigma_i((x_1,\ldots,x_m)) = (\max\{0,x_1\},\ldots, \max\{0,x_m\}),\quad i=1,2.$$ 
The activation function for the output is the
\emph{softmax function} 
$$\sigma_0((x_1,\ldots,x_m))=\frac{1}{\sum_{i=1}^m \exp(x_i)}\, (\exp(x_1),\ldots,\exp(x_m)).$$
We choose the \emph{categorical cross entropy loss function}
$$L(y, \hat y) = -\sum_{k=0}^{92} y_k\cdot\log(\hat y_k),$$
where $\hat y_k$ is the probability that there are $2k$ real solutions. In the process of training, the weights and biases are sequentially adjusted so that the loss function is minimized. 
To find the optimal weights and biases, we use stochastic gradient-based optimization.
Backpropagation calculates the gradient of the cost function with respect to the given weights \cite{Goodfellow-et-al-2016}. Our optimizitaion algorithm is \texttt{Adam} \cite{Adam}. \texttt{Adam} requires only first order gradients and computes individual adaptive learning rates. 
We use a batch size of 64.

We developed the model architecture during the data exploration phase. We tried several different architectures where we varied the number of layers, the number of neurons in each layer, and the activation functions. We found that networks with more than two hidden layers or fewer neurons had slower learning progress and resulted in worse accuracy on the validation set than our proposed model.

\subsection{Evaluation} \label{sec:evaluation}

We implemented the model outlined in \Cref{sec:our model} using TensorFlow \cite{tensorflow2015-whitepaper} with data sets $\mathcal{D}_1, \mathcal{D}_2$. We used an $80/20$ training-test-split and consider training our model in three ways: (1) using training data $\mathcal{D}_1$, (2) using training data $\mathcal{D}_2$ and (3) using training data $\mathcal{D}_1 \cup \mathcal{D}_2$. \Cref{tab:ml_eval} documents our results.

The values on the diagonal entries of \Cref{tab:ml_eval} are the validation results from training. All other results are computed on the whole set. We found that the accuracy on set $\mathcal{D}_1$ is very low, when training only on set $\mathcal{D}_2$ and vice versa which is not surprising considering how different the underlying distributions of $\mathcal{D}_1$ and $\mathcal{D}_2$ are. In addition, the purely random data from data set $\mathcal{D}_2$ is much harder to learn than $\mathcal{D}_1$.  
The model achieves validation accuracy of $97.47\%$ on $\mathcal{D}_1$ against a validation accuracy of only $47.33\%$ on $\mathcal D_2$. This means that when using $\mathcal{D}_2$ as training data, the model is over-fitting and learning random features from the data. There are many techniques to prevent this \cite{Goodfellow-et-al-2016} that can be explored in future work. Nevertheless, training on $\mathcal D_1$ works exceptionally well, even without any special techniques. 

\begin{table}
\begin{tabular}{|c|c|c|c|}
\hline 
    Training Data & Accuracy on $\mathcal{D}_1$ & Accuracy on $\mathcal{D}_2$ & Accuracy on $\mathcal{D}_1 \cup \mathcal{D}_2$\\\hline
    $\mathcal{D}_1$ & 95.59\% & 3.68\% & 49.88\%\\
    $\mathcal{D}_2$ & 3.53\% & 47.33\% & 45.69\%\\
    $\mathcal{D}_1 \cup \mathcal{D}_2$ & 90.76\% & 38.56\% & 60.20\% \\ \hline
\end{tabular}
\caption{The empirical results with the three different training sets} \label{tab:ml_eval}
\end{table}

It is reasonable to expect that $\mathcal D_1$ is easier to predict than $\mathcal D_2$, because the data set $\mathcal D_1$ represents a wider range of the behavior of the parameter space. Nevertheless, as $\mathcal D_1$ only contains parameters with up to 90 real tritangent circles, it does not capture the whole picture. We believe that the good performance of $\mathcal D_1$ is implied by the structure of the real discriminant from \Cref{thm:real_disc}. Learning the discriminant was approached in \cite{Bernal2020MachineLT}. It would be interesting to understand to what extent the real discriminant can be learned from $\mathcal D_1$.

\bibliographystyle{plain}
\bibliography{refs.bib}

\end{document}